\documentclass{amsart}

\usepackage{amssymb}
\usepackage{amsmath}
\usepackage{verbatim}
\usepackage{url}
\usepackage{graphicx}
\usepackage{color}
\usepackage{enumerate}

\newcommand{\opnm}{\operatorname}

\newcommand{\BigO}{\mathcal{O}}

\usepackage{mathrsfs}

\newcommand{\Hm}[1]{\leavevmode{\marginpar{\tiny%
$\hbox to 0mm{\hspace*{-0.5mm}$\leftarrow$\hss}%
\vcenter{\vrule depth 0.1mm height 0.1mm width \the\marginparwidth}%
\hbox to
0mm{\hss$\rightarrow$\hspace*{-0.5mm}}$\\\relax\raggedright #1}}}

\newtheorem{theorem}{Theorem}
\newtheorem*{theorem*}{Theorem}

\newtheorem{corollary}[theorem]{Corollary}

\theoremstyle{definition}

\theoremstyle{remark}

\newtheorem*{remark*}{Remark}
\newtheorem{example}[theorem]{Example}

\title[SVDs of certain solution operators]{Singular-value decomposition of solution operators to model evolution equations}

\author{Alexandru Aleman}

\email{aleman@maths.lth.se}

\address{Lund University, Mathematics, Faculty of Science, P.O. Box 118, S-221 00 Lund, Sweden}

\author{Joe Viola}

\email{Joseph.Viola@univ-nantes.fr}

\address{Laboratoire de Math\'ematiques J. Leray, UMR 6629 du CNRS, Universit\'e de Nantes, 2, rue de la
Houssini\`ere, 44322 Nantes Cedex 03, France}

\begin{document}

\begin{abstract}
We consider evolution equations generated by quadratic operators admitting a decomposition in creation-annihilation operators without usual ellipticity-type hypotheses; this class includes hypocoercive model operators.  We identify the singular value decomposition of their solution operators with the evolution generated by an operator of harmonic oscillator type, and in doing so derive exact characterizations of return to equilibrium and regularization for any complex time.
\end{abstract}

\maketitle

\section{Introduction}

We consider evolution equations given by operators acting on $L^2(\Bbb{R}^n)$ which can be decomposed into the annihilation and creation operators $A_j = \partial_{x_j} + x_j$ and $A_j^* = -\partial_{x_j} + x_j$. These operators are motivated by hypocoercive models; see Example \ref{ex.FP} below. For $M = (m_{jk})_{j,k=1}^n \in \Bbb{M}_{n\times n}(\Bbb{C})$, let
\begin{equation}\label{eq.def.P}
	P = \frac{1}{2}\sum_{j,k=1}^n m_{jk}A_k^* A_j.
\end{equation}
We present weak solutions to the problem
\begin{equation}\label{eq.evolution}
	\left\{\begin{array}{l} \partial_t u(t,x) + Pu(t,x) = 0,
	\\ u(0,x) = u_0(x) \in L^2(\Bbb{R}^n). \end{array}\right.
\end{equation}
Our goal is to precisely describe the norm of these weak solutions in a way which makes clear how the eigenvalues of $P$ determine the long-time behavior of these solutions.

We emphasize that we make no assumptions on the sign of $\Re \langle u, Pu\rangle$ or on any other version of ellipticity, assumptions which have been relied on previously even to define a solution to the evolution equation.  Nonetheless, in this much broader setting, we are able to establish a simple exact description of the decay and regularizing effects of these solutions. This sharpens and extends previous results on return to equilibrium and how long-term behavior is governed by the spectrum of the operator, and it puts these phenomena in an elementary dynamical setting.

The present work is a special case of a more general analysis of evolution equations generated by quadratic operators \cite{AlVi2014b}.  In particular, the special creation-annihilation operator form considered here is not stable under perturbations in the coefficients of a quadratic operator, but the more general analysis is applicable under a much weaker hypothesis.

We remark that necessary and sufficient conditions for a quadratic operator to admit such a creation-annihilation operator decomposition, after a unitary equivalence, may deduced from \cite[Thm.~1.4]{Vi2013} and \cite[Sec.~4]{AlVi2014b}.

Let $\{h_\alpha\}_{\alpha \in \Bbb{N}^n}$ be the orthonormal basis for $L^2(\Bbb{R}^n)$ formed by the Hermite functions, which may be realized as
\begin{equation}\label{eq.def.Hermite}
	h_\alpha(x) = \frac{1}{\sqrt{2^{|\alpha|} \alpha! \sqrt{\pi^n}}}(A^*)^\alpha e^{-x^2/2}
\end{equation}
for $(A^*)^\alpha = (A_1^*)^{\alpha_1}\cdots (A_n^*)^{\alpha_n}$ defined via the multi-index $\alpha$ and the $n$-vector of creation operators. An induction argument and the commutator identity $[A_j, A_k^*] = 2\delta_{jk}$ shows that
\begin{equation}\label{eq.cr.ann.hermite}
	A_k^* A_j h_\alpha = 2\sqrt{\alpha_j\alpha_k}h_{\alpha - e_j + e_k},
\end{equation}
with $\alpha - e_j + e_k$ obtained by decreasing the $j$th index and increasing the $k$th index of $\alpha$. We recall that therefore the Hermite functions diagonalize the standard harmonic oscillator
\begin{equation}\label{eq.def.Q0}
	Q_0 = \frac{1}{2}\sum_{j=1}^n A_j^* A_j = \frac{1}{2}(-\Delta + |x|^2 - n),
\end{equation}
here chosen so that $\opnm{Spec} Q_0 = \Bbb{N}$:
\[
	Q_0 h_\alpha = |\alpha|h_\alpha, \quad \forall \alpha \in \Bbb{N}^n.
\]

We note that
\[
	\begin{aligned}
	2[Q_0, A_k^*A_j] &= \sum_{\ell = 1}^n [A_\ell^*A_\ell, A_k^* A_j]
	\\ &= \sum_{\ell = 1}^n \left(A_\ell^*[A_\ell, A_k^*] A_j + A_k^*[A_\ell^*, A_j] A_\ell\right)
	\\ &= \sum_{\ell = 1}^n 2\left(\delta_{\ell k} A_\ell^* A_j - \delta_{\ell j}A_k^* A_\ell\right)
	\\ &= 0,
	\end{aligned}
\]
and therefore $P$ commutes with $Q_0$. We define the natural energy subspaces
\begin{equation}\label{eq.def.Em}
	E_m = \opnm{span}\{h_\alpha \::\: |\alpha| = m\} = \ker(Q_0 - m),
\end{equation}
which are $P$-invariant since they are eigenspaces of an operator which commutes with $P$.  By putting $P|_{E_m}$ into Jordan normal form, it is therefore elementary that $P$ has a family of generalized eigenfunctions with span identical to that of the Hermite functions, which is dense in $L^2(\Bbb{R}^n)$.  These eigenfunctions generally are not orthogonal, however, and their associated projections grow up to exponentially rapidly \cite[Cor.~1.6]{Vi2013}.

\begin{example}\label{ex.FP}
The Fokker-Planck model \cite[Sec.~5.5.1]{HeNiBook} for a quadratic potential
\[
	P_a = \frac{1}{2}\left(x_2^2 - \partial_{x_2}^2-1\right) + a\left(x_1\partial_{x_2} - x_2\partial_{x_1}\right)
\]
is of the form \eqref{eq.def.P} with
\[
	M_a = \left(\begin{array}{cc} 0 & -a \\ a & 1\end{array}\right).
\]
\end{example}

\section{Results}

Our main result is an exact simple description of the norm behavior of a weak solution to \eqref{eq.evolution} relying only on the singular value decomposition of the invertible matrix $e^{-tM}$.

\begin{theorem*}\label{thm.only}
Let $P$ be as in \eqref{eq.def.P}, fix $t \in \Bbb{C}$, and let 
\[
	\begin{aligned}
	0 < \sigma_1 \leq \sigma_2 \leq \dots \leq \sigma_n,
	\\ \{\sigma_j^2\}_{j=1}^n = \opnm{Spec}((e^{-t M})^*e^{-t M}),
	\end{aligned}
\]
be the singular values of $e^{-t M}$, repeated for multiplicity.

There exists a closed densely defined operator, which we denote $\exp(-tP)$, extending the solution to \eqref{eq.evolution} defined on $\opnm{span}\{h_\alpha\}_{\alpha \in \Bbb{N}^n}$. There exist furthermore two unitary operators $\mathcal{U}_1$ and $\mathcal{U}_2$ on $L^2(\Bbb{R}^n)$, which preserve each $E_m$ from \eqref{eq.def.Em}, reducing $\exp(-tP)$ to a solution operator corresponding to a sum of harmonic oscillators: specifically,
\begin{equation}\label{eq.SVD}
	\mathcal{U}_1 \exp(-tP) \mathcal{U}_2^* = \exp\left(\frac{1}{2}\sum_{j=1}^n (\log \sigma_j)(x_j^2 - \frac{\partial^2}{\partial x_j^2} - 1)\right).
\end{equation}
\end{theorem*}

\begin{remark*}
We say that this is a singular value decomposition of the solution operator $\exp(-tP)$ because the operator 
\begin{equation}\label{eq.def.Q.sigma}
	Q = Q(\sigma_1,\dots, \sigma_n) = \frac{1}{2}\sum_{j=1}^n(\log \sigma_j)(x_j^2 - \frac{\partial^2}{\partial x_j^2} - 1)
\end{equation}
exponentiated on the right-hand side of \eqref{eq.SVD} is self-adjoint and diagonal with respect to the Hermite basis:
\[
Q h_\alpha = \left(\sum_{j=1}^n(\log \sigma_j)\alpha_j\right)h_\alpha.
\]
Therefore $\exp Q$ is self-adjoint, positive definite, and diagonal with respect to the Hermite basis.
\end{remark*}

\begin{proof}
The main tool here is the classical Bargmann transform; to fix a reference, we refer the reader to \cite[Sec.~1.6,~1.7]{FoBook}, but some changes of variables are necessary to adjust for factors of $\sqrt{\pi}$. We recall that the Bargmann transform
\[
	\mathcal{B}f(z) = \pi^{-3n/4}\int_{\Bbb{R}^n} f(x)e^{\sqrt{2}xz-x^2/2 - z^2/2}\,dx
\]
is a unitary map from $L^2(\Bbb{R}^n)$ to the Fock space $\mathcal{F}$ consisting of holomorphic functions of $n$ variables for which the norm
\[
	\|v\|_\mathcal{F} = \left(\int_{\Bbb{C}^n} |v(z)|^2 e^{-|z|^2}\,dL(z)\right)^{1/2}
\]
is finite. Here $dL(z)$ is Lebesgue measure on $\Bbb{C}^n \sim \Bbb{R}^n_{\Re z} \times \Bbb{R}^n_{\Im z}$.

We have that
\[
	\mathcal{B}A_j\mathcal{B}^* = \sqrt{2}\partial_{z_j}\textnormal{  and  }\mathcal{B}A_j^*\mathcal{B}^* = \sqrt{2}z_j,
\]
where the derivative is holomorphic in the former and the latter is a multiplication operator. Therefore, for $P$ in \eqref{eq.def.P}, we have
\begin{equation}\label{eq.def.P.Fock}
	\mathcal{B}P\mathcal{B}^* = (Mz)\cdot \partial_z.
\end{equation}
Furthermore,
\[
	\mathcal{B}h_\alpha = \frac{1}{\sqrt{\pi^n \alpha!}}z^\alpha,
\]
so the $E_m$ in \eqref{eq.def.Em} map to spaces of homogeneous polynomials
\[
	\mathcal{E}_m := \mathcal{B}E_m = \opnm{span}\{z^\alpha \::\: |\alpha| = m\}.
\]

Calculating that $\partial_z (u(Fz)) = F^{\top} (\partial_z u)(Fz)$ for any matrix $F \in \Bbb{M}_{n\times n}(\Bbb{C})$, we have that
\begin{equation}\label{eq.solve.Fock}
	v(t,z) = \exp(-t(Mz)\cdot \partial_z)v_0(z) := v_0(e^{-tM}z)
\end{equation}
solves
\begin{equation}\label{eq.evolution.Fock}
	\left\{\begin{array}{l} \partial_t v(t,x) + ((Mz)\cdot \partial_z)v(t,x) = 0,
	\\ v(0,x) = v_0(x) \in \mathcal{F}. \end{array}\right.
\end{equation}
The solution is unique in the space of holomorphic functions because any solution must obey $\partial_t(v(t, e^{-tM}z)) = 0$ and therefore $v(t,e^{-tM}z) = v_0(z)$.  Therefore the solution agrees with the realization as a matrix exponential on any element of $\opnm{span}\{\mathcal{B}h_\alpha\} = \opnm{span}\{z^\alpha\}$, the polynomials.  This operator $\exp(-t(Mz)\cdot \partial_z)$ is densely defined because the polynomials are dense in $\mathcal{F}$; equivalently, the Hermite functions have dense span in $L^2(\Bbb{R}^n)$. It has a closed graph when equipped with the domain
\[
	\operatorname{Dom}(\exp(-t(Mz\cdot\partial_z))) = \{v \in \mathcal{F} \::\: v(e^{-tM}z) \in \mathcal{F}\}
\]
because convergence in $\mathcal{F}$ implies convergence in $L^2_{\textnormal{loc}}(\Bbb{C}^n)$ which implies pointwise convergence for holomorphic functions. Therefore, whenever $\{v_k\}_{k\in \Bbb{N}}$ is a sequence in $\operatorname{Dom}(\exp(-t(Mz\cdot\partial_z)))$ for which
\[
	(v_k, \exp(-t(Mz)\cdot \partial_z)v_k) \to (v_\infty, w)
\]
in $\mathcal{F}\times \mathcal{F}$, we have that $w(z) = v_\infty(e^{-tM}z) = \exp(-t(Mz)\cdot \partial_z)v_\infty(z)$.

Naturally, we define
\[
	\exp(-tP) = \mathcal{B}^*\exp(-t(Mz)\cdot \partial_z)\mathcal{B}.
\]

A change of variables makes it clear that, whenever $U \in \Bbb{M}_{n\times n}(\Bbb{C})$ is a unitary matrix,
\[
	\mathcal{V}_Uv(z) = v(Uz)
\]
is a unitary transformation on $\mathcal{F}$ preserving each $\mathcal{E}_m$.  Furthermore, $\mathcal{V}_U^* = \mathcal{V}_{U^*}$. For $t \in \Bbb{C}$ fixed, by the singular value decomposition for $e^{-tM}$ we let $U_1, U_2$ be unitary matrices such that
\[
	e^{-tM} = U_2\Sigma U_1^*
\]
for $\Sigma$ the diagonal matrix with entries $\sigma_1,\dots, \sigma_n$.

Then
\[
	\begin{aligned}
	\mathcal{V}_{U_1}\exp(-t(Mz)\cdot \partial_z)\mathcal{V}_{U_2}^*v_0(z) &= v_0(U_2^* e^{-tM}U_1 z)
	\\ &= v_0(\Sigma z)
	\\ &= \exp((\log \Sigma)z\cdot \partial_z)v_0(z),
	\end{aligned}
\]
where the formula for $\exp((\log \Sigma)z\cdot \partial_z)$ comes from the same reasoning which revealed that \eqref{eq.solve.Fock} solves \eqref{eq.evolution.Fock}.  But by \eqref{eq.def.P.Fock} we have that 
\[
	\mathcal{B}^*\left((\log \Sigma)z \cdot \partial_z \right)\mathcal{B} = \frac{1}{2}\sum_{j=1}^n (\log \sigma_j)(x_j^2 - \frac{\partial^2}{\partial x_j^2} - 1).
\]
Therefore, writing $\mathcal{U}_j = \mathcal{B}^* \mathcal{V}_{U_j}\mathcal{B}$ and noting that the operators $\mathcal{U}_j$ preserve the spaces $E_m$ because the operators $\mathcal{V}_{U_j}$ preserve the spaces $\mathcal{E}_m$, we have proven the theorem.
\end{proof}

We continue by proving several immediate consequences, the first of which is the identification of the eigenvectors of $P$. These form a complete system, meaning that their span is dense in $L^2(\Bbb{R}^n)$.

\begin{corollary}\label{cor.spectrum}
Let $\{\lambda_j\}_{j=1}^n$ be the eigenvalues of $M$, repeated for algebraic multiplicity. Let $G$ be such that $GMG^{-1}$ is in Jordan normal form, therefore having diagonal entries $\lambda_1,\dots,\lambda_n$.  Then
\[
	\left\{\mathcal{B}^*\left((Gz)^\alpha\right)\right\}_{\alpha \in \Bbb{N}^n}
\]
forms a complete system of eigenfunctions of $P$ in \eqref{eq.def.P} with corresponding generalized eigenvalues 
\[
	\lambda_\alpha = \sum_{j=1}^n \lambda_j\alpha_j.
\]
\end{corollary}

\begin{proof}
This essentially follows the part of the proof of \cite[Thm.~3.5]{Sj1974}; see also \cite[Lem.~4.1]{HiSjVi2013}. It suffices to show that $\{(Gz)^\alpha\}$ are eigenfunctions of $Mz\cdot \partial_z$; using that $\partial_z (v(Gz)) = G^\top (\partial_z v)(Gz)$, we see that
\[
	(Mz\cdot \partial_z)(Gz)^\alpha = \left. (GMG^{-1}\zeta\cdot \partial_\zeta) \zeta^\alpha \right|_{\zeta = Gz}.
\]
Letting $e_j$ indicate the standard basis vector with $1$ in the $j$-th component and $0$ elsewhere,
\[
	(GMG^{-1}\zeta\cdot \partial_\zeta)\zeta^\alpha = \sum_{j=1}^n \alpha_j(\lambda_j \zeta^\alpha + \gamma_j \zeta^{\alpha-e_j + e_{j-1}})
\]
with $\gamma_j \in \{0,1\}$ equaling 1 only when $e_j$ is part of a Jordan block of $GMG^{-1}$.  In particular, if $\gamma_j \neq 0$ then $\lambda_{\alpha - e_j + e_{j-1}} = \lambda_\alpha$.

Therefore
\[
	(GMG^{-1}\zeta\cdot \partial_\zeta - \lambda_\alpha)\zeta^\alpha = \sum c_\beta \zeta^\beta,
\]
where $c_\beta \neq 0$ is possible only when $|\beta| = |\alpha|$, when $\lambda_\beta = \lambda_\alpha$, and when $\sum_{j=1}^n j\alpha_j < \sum_{j=1}^n j\beta_j$.  This can only be repeated finitely many times before one violates the trivial bound $\sum_{j=1}^n j\beta_j \leq n|\beta|$.  Since also $\sum_{j=1}^n j\alpha_j \geq |\alpha|$, we conclude that
\[
	(GMG^{-1}\zeta\cdot \partial_\zeta - \lambda_\alpha)^{(n-1)|\alpha|+1}\zeta^\alpha = 0, \quad \forall \alpha \in \Bbb{N}^n,
\]
proving that $\zeta^\alpha = (Gz)^\alpha$ is a generalized eigenvector of $Mz\cdot \partial_z$ as claimed. Since $\opnm{span} \{\zeta^\alpha\} = \opnm{span} \{z^\alpha\}$ which is dense in $\mathcal{F}$, these generalized eigenvectors form a complete family, completing the proof of the corollary.
\end{proof}

\begin{corollary}\label{cor.bounded}
The operator $\exp(-tP)$ extends to a bounded operator if and only if $\|e^{-tM}\| \leq 1$ and is compact if and only if $\|e^{-tM}\| < 1$.
\end{corollary}

\begin{proof}
For $\exp(-tP)$ to be bounded (resp.\ compact), it is necessary and sufficient for the self-adjoint operator $\exp Q$ for $Q$ as in \eqref{eq.def.Q.sigma} to be bounded (resp.\ compact), and the singular values of $\exp(-tP)$ are the eigenvalues of $\exp Q$. From the Hermite function diagonalization,
\begin{equation}\label{eq.Spec.exp.Q}
	\opnm{Spec}(\exp Q) = \left\{\prod_{j=1}^n \sigma_j^{\alpha_j} \::\: \alpha \in \Bbb{N}^n\right\}.
\end{equation}
Since the $\sigma_j$ are arranged in increasing order, it is necessary and sufficient for boundedness to have $\sigma_n = \|e^{-t M}\| \leq 1$, because otherwise 
\begin{equation}\label{eq.norm.maximizers}
	\|\exp(-tP)\mathcal{U}_2^*h_{(0,\dots,0,k)}\| = \sigma_n^k
\end{equation}
tends to infinity exponentially rapidly as $k\to\infty$. Similarly, if $\sigma_n = \|e^{-tM}\| = 1$, then \eqref{eq.norm.maximizers} shows that the eigenvalue 1 for $\exp Q$ has infinite multiplicity, precluding compactness of $\exp(-tP)$. If $\sigma_n = \|e^{-tM}\| < 1$, then $\exp Q$ is a positive self-adjoint operator whose eigenvalues tend to zero, and therefore $\exp Q$ and $\exp(-tP)$ are compact.
\end{proof}

\begin{corollary}
	The operator $\exp(-tP)$ is bounded for all $t \geq 0$ simultaneously if and only if
	\[
		\Re \langle Mz, z\rangle \geq 0, \quad \forall |z| = 1,
	\]
	and for $\exp(-tP)$ to be compact for all $t > 0$, it is sufficient that
	\[
		\Re \langle Mz, z\rangle > 0, \quad \forall |z| = 1.
	\]
\end{corollary}

\begin{remark*}
	The symbol of $P$ as a differential operator is
	\[
		\frac{1}{2}M(x-i\xi)\cdot (x+i\xi), \quad (x,\xi) \in \Bbb{R}^{2n}.
	\]
	Setting $z = x-i\xi$, we see that the conditions in the corollary are equivalent to the classical ellipticity conditions that the real part of the symbol should be positive (semi-)definite.

	The second condition is not necessary to have $\exp(-tP)$ compact for all $t > 0$, as is clear from Examples \ref{ex.FP} and \ref{ex.FP.cont}.
\end{remark*}

\begin{proof}
	We compute that
	\[
		|e^{-tM}z|^2 = |z|^2 - 2t\Re\langle Mz, z\rangle + \BigO(t^2).
	\]
	Therefore
	\[
		\frac{d}{dt}|e^{-tM}z|^2 = -2\Re \langle Mz, z\rangle,
	\]
	and note that $\|e^{-tM}\| \leq 1$ for all $t \geq 0$ if and only if $|e^{-tM}z|^2$ is decreasing on $\{|z| = 1\}$ while $\|e^{-tM}\| < 1$ for all $t > 0$ if $|e^{-tM}z|^2$ is strictly decreasing on $\{|z| = 1\}$.  This, along with Corollary \ref{cor.bounded}, proves the corollary.
\end{proof}

We next show precise asymptotics for the exponential decay, like $\exp(-Cj)^{1/n}$, of the singular values $s_j$ of $\exp(-t P)$ whenever this operator is compact.  In particular, this implies that $\exp(-tP)$ is in every Schatten class $\mathfrak{S}_p$, $p \in (0, \infty)$, as soon as it is compact.

\begin{corollary}
Let $P$ and $\exp(-tP)$ be as in the Theorem, and let $\Sigma$ be the diagonal matrix with entries equal to the singular values of $e^{-tM}$. If $\exp(-tP)$ is compact, then its singular values, arranged in non-increasing order, obey
\[
	-\log s_j = \left(n! \det(-\log \Sigma)\right)^{1/n} j^{1/n} (1+\BigO(j^{-1/n})).
\]
\end{corollary}

\begin{proof}
From \eqref{eq.Spec.exp.Q}, we have the singular values of $\exp(-tP)$ which tend to zero by the assumption that $\exp(-tP)$ is compact. Note that this means that $\sigma_j \in (0,1)$ for all $j$. Taking logarithms, we see that the singular value $s_j$ of $\exp(-tP)$ is determined by the requirements that
\[
	 \# \{\alpha \in \Bbb{N}^n \::\: \sum_{k=1}^n (-\log \sigma_k)\alpha_k < -\log s_j\} < j
\]
and that
\[
	\# \{\alpha \in \Bbb{N}^n \::\: \sum_{k=1}^n (-\log \sigma_k)\alpha_k = -\log s_j\} = j.
\]
An elementary argument (cf.~\cite[Eq.~(68)]{Vi2012b}) comparing the cardinality of this set with the volume of the boxes anchored at lattice points --- which is in turn well-approximated by the volume of a simplex --- shows that as $R \to \infty$ we have the estimate
\[
	\# \{\alpha \in \Bbb{N}^n \::\: -\sum_{k=1}^n (-\log \sigma_k)\alpha_k \leq R\} = \frac{R^n}{n! \prod_{k=1}^n (-\log \sigma_k)}(1+\BigO(R^{-1})).
\]
Letting $R = -\log s_j$ and taking $n$-th roots, we obtain
\[
	(-\log s_j)(1+\BigO(-\log s_j)) = (n!\det(-\log \Sigma))^{1/n}j^{1/n}.
\]
Since we now have that $-\log s_j$ grows like $j^{1/n}$, we have that $(1+\BigO(-\log s_j))^{-1} = 1+\BigO(j^{-1/n})$ as $j \to \infty$, and the result follows.
\end{proof}

We can study the regularization properties of the operator $\exp(-tP)$ by comparing it with the standard harmonic oscillator semigroup through composition.

\begin{corollary}
Let $Q_0$ be the standard harmonic oscillator \eqref{eq.def.Q0}. Then, for $\delta \in \Bbb{R}$,
\begin{equation}\label{eq.smoothing.l}
	\exp(\delta Q_0) \exp(-tP)
\end{equation}
and
\begin{equation}\label{eq.smoothing.r}
	\exp(-tP)\exp(\delta Q_0)
\end{equation}
are bounded on $L^2(\Bbb{R}^n)$ if and only if
\[
	\|e^{-tM}\| \leq e^{-\delta}.
\]
\end{corollary}

\begin{proof}
From the proof of the Theorem, for any $v \in \mathcal{F}$,
\[
	\mathcal{B}\exp(\delta Q_0) \exp(-tP)\mathcal{B}^*v(z) = \mathcal{B}\exp(-tP)\exp(\delta Q_0)\mathcal{B}^*v(z) = v(e^{\delta-tM}z).
\]
But again this map is bounded if and only if $\|e^{\delta - tM}\|\leq 1$, which gives the corollary.
\end{proof}

\begin{corollary}
Let $\exp(-tP)$ be as in the Theorem and suppose that $\|e^{-tM}\| \leq 1$. For the Hermite functions $\{h_\alpha\}$ from \ref{eq.def.Hermite}, let
\[
	\Pi_N u = \sum_{|\alpha| \leq N} \langle u, h_\alpha\rangle h_\alpha
\]
be the orthogonal projection onto $\bigoplus_{k=0}^N E_m$ for $E_m$ in \eqref{eq.def.Em}. Then, for all $N \in \Bbb{N}$,
\[
	\|\exp(-tP)(1-\Pi_N)\|_{\mathcal{L}(L^2(\Bbb{R}^n))} = \|e^{-tM}\|^{N+1}.
\]
\end{corollary}

\begin{remark*}
When $N = 0$, the decay of $\exp(-tP)(1-\Pi_0)$ is known as return (or convergence) to equilibrium; see e.g.\ \cite[Ch.~6]{HeNiBook}.
\end{remark*}

\begin{proof}
The norm of an operator is the largest of its singular values, and from the classical formula for the eigenvalues of an operator of harmonic oscillator type and the Theorem, the singular values of $\exp(-tP)|_{E_m}$ are
\[
	\left\{\prod_{j=1}^n \sigma_j^{\alpha_j} \::\: |\alpha| = m\right\}.
\]
We recall that every $\sigma_j \in (0,1]$ since $\|e^{-tM}\| \leq 1$, and therefore the largest such singular value appears when $m = N+1$ with singular value $\sigma_n^{N+1} = \|e^{-tM}\|^{N+1}$. This identifies the norm of the operator composed with the projection as desired.
\end{proof}

\begin{corollary}
If $\opnm{Spec} M \subset \{\Re \lambda > 0\}$ and $\exp(-tP)$ is as in the Theorem then there exists $T_0 > 0$ such that $\exp(-tP)$ is compact on $L^2(\Bbb{R}^n)$ for all $t > T_0$.  Moreover, for any $\delta > 0$, however large, there exists $T_\delta > 0$ such that the operators \eqref{eq.smoothing.l} and \eqref{eq.smoothing.r} are compact on $L^2(\Bbb{R}^n)$ for all $t > T_\delta$.
\end{corollary}

\begin{proof}
Both statements follow immediately from the fact that $\|e^{-tM}\| \to 0$ which follows from exponentiating the Jordan normal form of $M$. In fact, if 
\[
	\alpha = \min_{\lambda \in \opnm{Spec} M}\Re \lambda
\]
and $r$ is the size of the largest Jordan block corresponding to a $\lambda \in \opnm{Spec} M$ for which $\Re \lambda = \alpha$, then there exists $C > 0$ where, for $t$ sufficiently large,
\begin{equation}\label{eq.exptM.long.time}
	\frac{1}{C}t^{r-1}e^{-t\alpha} \leq \|e^{-tM}\| \leq Ct^{r-1}e^{-t\alpha}.
\end{equation}
We may therefore certainly take $T_\delta \leq C_0(1+\delta)$ for some $C_0 > 0$ and all $\delta > 0$.
\end{proof}

\section{Further directions and example details}

For reasons of length and in order to rest entirely on the classical Bargmann transform, the current work avoids an in-depth treatment of many natural questions. In \cite{AlVi2014b}, we continue in these directions, considering a more stable class of quadratic operators by using a family of FBI-Bargmann transforms. We also consider many other questions and extend the analysis here, including the question of whether the eigenfunctions form a core for the solution operators, relationships between regularizing properties of different harmonic oscillator semigroups, and the way ellipticity and weak ellipticity (from a bracket condition) reappear in an elementary way in the Taylor expansion of, for instance, $\|e^{-tM}\|$. Throughout, we see a theme where the range of the symbol determines behavior for short times and the eigenvalues determine behavior in long times, even when the short-time behavior is wildly unbounded.

\begin{example}\label{ex.FP.cont}
We illustrate these results on the Fokker-Planck model $P_a$ from Example \ref{ex.FP}.

Return to equilibrium results like those in \cite[Thm.~3]{GaMi2013} follow from analysis of the norm $\|e^{-tM_a}\|$. The phenomenon of weak ellipticity is reflected geometrically in the fact that the norm is decreasing along integral curves of $\frac{d}{dt}z(t) = Mz(t)$, but the integral curves are tangent to the unit circle at $\{z_2 = 0\}$. More precisely, $|e^{-tM_a}(z_1, z_2)|$ is strictly decreasing for all $(z_1, z_2)\in \Bbb{C}^2\backslash \{0\}$, but only slowly for small $t$: that is, while
\[
	z_2 \neq 0 \implies \frac{d}{dt}|e^{-tM_a}(z_1, z_2)| < 0,
\]
decay along the $z_1$-axis is slower:
\[
	|e^{-tM_a}(1,0)| = 1-\frac{a^2}{3}t^3 + \BigO(t^4).
\]
Some further calculations reveal that
\[
	\|e^{-tM_a}\| = 1-\frac{a^2}{12}t^3 + \BigO(t^4),
\]
with $z = (1, at/2)$ realizing the maximum up to an error of $\BigO(t^4)$.

That regularization and return to equilibrium of $\exp(-tP_a)$ are governed for large times by $\opnm{Spec} M_a$ is elementary because $\opnm{Spec} M_a$ (almost) determines $\|e^{-tM_a}\|$ as $t\to \infty$, as seen in \eqref{eq.exptM.long.time}.

In addition, a weak definition of the solution operator remains even if we add a perturbation which destroys the ellipticity: let
\[
	\tilde{P}_{a,b} = P_a - b(x_1^2 - \partial_{x_1}^2), \quad a\in \Bbb{R}, b > 0.
\]
Having now 
\[
	\tilde{M}_{a,b} = \left(\begin{array}{cc} -b & -a \\ a & 1\end{array}\right),
\]
it is clear that
\[
	|e^{-t\tilde{M}_{a,b}}(1,0)| = 1 + bt + \BigO(t^2)
\]
and therefore $\exp(-t\tilde{P}_{a,b})$ is unbounded for small times. Additionally, a simple scaling argument and a classical pseudomode construction taken from \cite{DeSjZw2004} shows that $\opnm{Spec} \tilde{P}_{a,b} = \Bbb{C}$; see \cite[Sec.~3]{AlVi2014b}.

On the other hand, so long as $b > 0$ is sufficiently small that $\opnm{Spec}\tilde{M}_{a,b} \subset \{\Re \lambda > 0\}$, which means that $b < \min\{a^2, 1\}$, we have $\|e^{-tM_{a,b}}\| \to 0$ exponentially rapidly as $t \to \infty$.  Therefore, for sufficiently large times, the weakly defined solution operator exhibits all the same properties of compactness, exponentially decaying singular values, regularization, and exponentially fast return to equilibrium which are enjoyed by $P_a$.

\begin{figure}
\centering
	\includegraphics[width=0.8\textwidth]{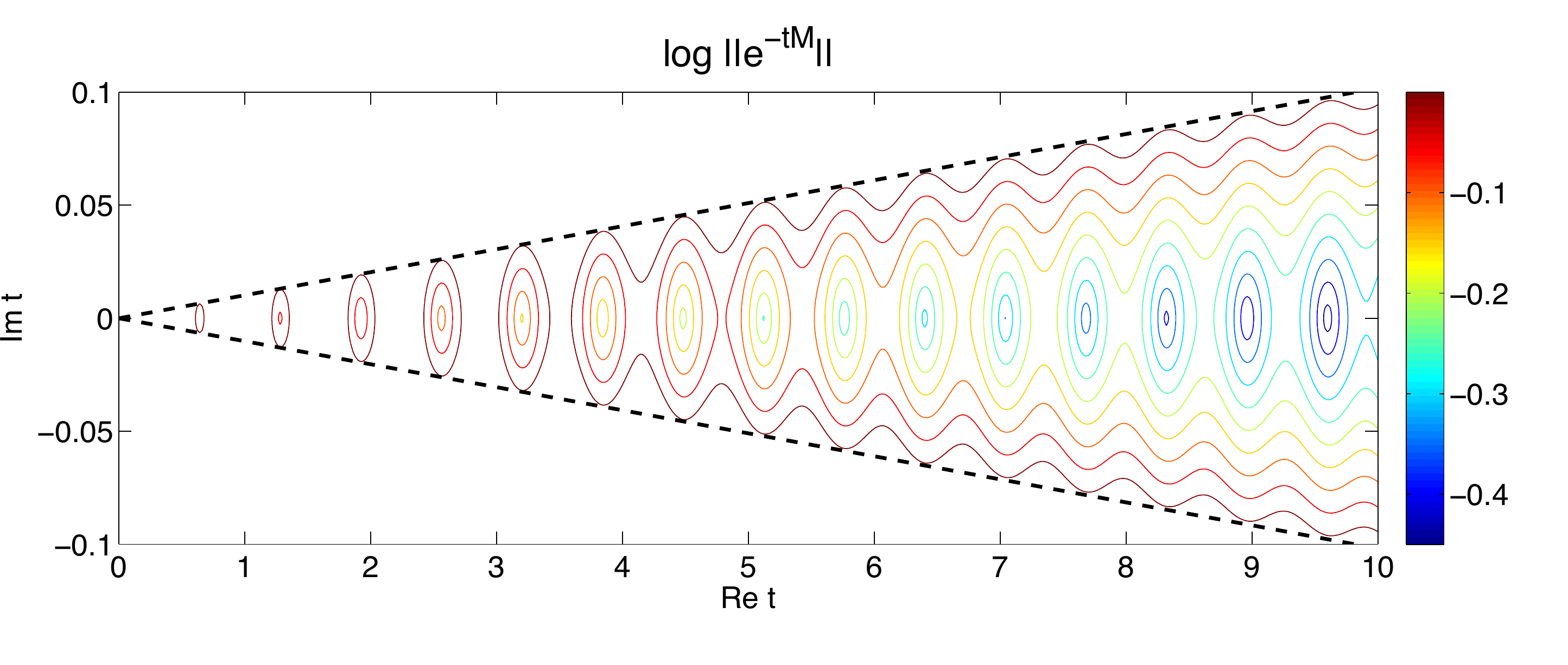}
	\caption{Region where $\exp(-t\tilde{P}_{a,b})$ is bounded and/or compact for $a = 5, b = 0.9$.}
	\label{fig.only}
\end{figure}
\end{example}

As a final illustration, we let $a = 5$ and $b = 0.9$ and present in Figure \ref{fig.only} the contours of $\log \|e^{-t\tilde{M}_{a,b}}\|$ when $\|e^{-t\tilde{M}_{a,b}}\| \leq 1$. Therefore $\exp(-t\tilde{P}_{a,b})$ is bounded inside the outermost curve and return to equilibrium is stronger as $\log \|e^{-t\tilde{M}_{a,b}}\|$ becomes more negative. Previously, it was not clear that $\exp(-t\tilde{P}_{a,b})$ could be defined as a bounded operator for any $t > 0$; now boundedness for large $t > 0$ is obvious from the spectrum of $\tilde{M}_{a,b}$. Geometrically, it is intuitively clear that the strong rotation combined with expansion in the $z_1$ direction allows integral curves of $\dot{z} = -Mz$ to exit and re-enter the unit ball multiple times, and this is reflected in the isolated regions of boundedness for $\exp(-t\tilde{P}_{a,b})$.

Finally, we see clearly that the eigenvalues determine, to a large extent, whether $\exp(-t\tilde{P}_{a,b})$ is bounded for $t\in\Bbb{C}$ and $|t|$ large: if $A$ were a normal operator with the same eigenvalues as $\tilde{P}_{a,b}$ indicated by Corollary \ref{cor.spectrum}, then the sector demarcated by dotted lines is precisely the set of $t \in \Bbb{C}$ for which $\exp(-tA)$ would be a bounded operator. Since the boundedness of $\exp(-tA)$ for a normal operator $A$ is determined solely by its eigenvalues, we see that the set of $t\in\Bbb{C}$ for which $\exp(-t\tilde{P}_{a,b})$ is bounded strongly resembles, for $|t|$ large, the set indicated when considering only the eigenvalues of $\tilde{P}_{a,b}$.

\bibliographystyle{acm}
\bibliography{MicrolocalBibliography2}

\end{document}